\documentclass[11pt]{amsart}

\makeindex

\usepackage{amsfonts,graphics,amsmath,amsthm,amsfonts,amscd,amssymb,amsmath,latexsym,multicol,euscript}
\usepackage{epsfig,url}
\usepackage{flafter}
\usepackage{hyperref}
\usepackage[all,cmtip,line]{xy}
\usepackage{array}
\usepackage[english]{babel}
\usepackage{overpic}
\usepackage{subfig}
\usepackage{multirow}

\usepackage{pgf,tikz}
\usepackage{mathrsfs}
\usetikzlibrary{arrows}

\usepackage{wrapfig}

\usepackage{enumerate}

\definecolor{qqqqff}{rgb}{0,0,1}
\definecolor{xdxdff}{rgb}{0.49,0.49,1}
\definecolor{uuuuuu}{rgb}{0.27,0.27,0.27}
\definecolor{cqcqcq}{rgb}{0.75,0.75,0.75}

\newtheorem{theorem}{Theorem}[section]
\newtheorem{main theorem}{Main Theorem}
\newtheorem{proposition}[theorem]{Proposition}

\newtheorem{corollary}[theorem]{Corollary}

\newtheorem{lemma}[theorem]{Lemma}

\newtheorem{theorem*}{Theorem}
\newtheorem{corollary*}[theorem*]{Corollary}
\newtheorem{conjecture*}[theorem*]{Conjecture}

\theoremstyle{definition}
\newtheorem{example}[theorem]{Example}
\newtheorem{definition}[theorem]{Definition}
\newtheorem{definition-lemma}[theorem]{Definition-Lemma}

\newtheorem{remark}[theorem]{Remark}

\newcommand\R{\mathbb{R}}
\newcommand\Q{\mathbb{Q}}
\newcommand\Z{\mathbb{Z}}

\renewcommand\P{\mathbb{P}}

\newcommand\eps{\varepsilon}
\renewcommand\epsilon{\varepsilon}

\newcommand{\mc}{\mathcal}

\newcommand{\ol}[1]{\overline{#1}}

\DeclareMathOperator{\N^1}{N^1}

\DeclareMathOperator{\ord}{ord}
\DeclareMathOperator{\mult}{mult}
\DeclareMathOperator{\Int}{Int}

\DeclareMathOperator{\Supp}{Supp}

\DeclareMathOperator{\vol}{vol}
\DeclareMathOperator{\val}{val}

\DeclareMathOperator{\Nef}{Nef}

\DeclareMathOperator{\Eff}{Eff}

\DeclareMathOperator{\Bigdiv}{Big}

\DeclareMathOperator{\Null}{Null}

\DeclareMathOperator{\SC}{SC}   

\newcommand{\bm}{\mathbf B_-}  

\newcommand{\bp}{\mathbf B_+}  
\newcommand{\okbd}{\Delta}
\newcommand{\okval}{\Delta^{\val}}
\newcommand{\oklim}{\Delta^{\lim}}

\newcommand{\kappanu}{\kappa_\nu}

\begin{document}

\title{Okounkov bodies and Zariski decompositions on surfaces}

\author{Sung Rak Choi}
\address{Department of Mathematics, Yonsei University, Seoul, Korea}
\email{sungrakc@yonsei.ac.kr}

\author{Jinhyung Park}
\address{School of Mathematics, Korea Institute for Advanced Study, Seoul, Korea}
\email{parkjh13@kias.re.kr}

\author{Joonyeong Won}
\address{Department of Mathematical Sciences, KAIST, Daejeon, Korea}
\curraddr{Center for Geometry and Physics, Institute for Basic Science, Pohang, Korea}
\email{leonwon@ibs.re.kr}

\date{\today}

\begin{abstract}
The purpose of this paper is to investigate the close relation between Okounkov bodies and Zariski decompositions of pseudoeffective divisors on smooth projective surfaces. Firstly, we completely determine the limiting Okounkov bodies on such surfaces, and give applications to Nakayama constants and Seshadri constants. Secondly, we study how the shapes of Okounkov bodies change as we vary the divisors in the big cone.
\end{abstract}

\maketitle


\section{Introduction}

To a big divisor $D$ on a variety, one can associate a convex body $\okbd_{Y_\bullet}(D)$ with respect to an admissible flag $Y_\bullet$ called the \emph{Okounkov body}.
Inspired by the works of Okounkov in \cite{O1}, \cite{O2}, Lazarsfeld-Musta\c{t}\u{a} (\cite{lm-nobody}) and Kaveh-Khovanskii (\cite{KK}) initiated the systematic study of the Okounkov bodies of big divisors.
In \cite{chpw1}, two natural ways to associate convex bodies to a pseudoeffective divisor $D$ with respect to an admissible flag $Y_\bullet$ were introduced. They are called the \emph{limiting Okounkov body} $\oklim_{Y_\bullet}(D)$ and the \emph{valuative Okounkov body} $\okval_{Y_\bullet}(D)$.
It was proved that some of the fundamental properties of divisors are encoded in these convex bodies.
We refer to Section \ref{sec3} for the definitions and basic properties of the Okounkov bodies.

The study on Okounkov bodies follows a simple philosophy that the structure of the Okounkov bodies should tell us the information of the divisors. Thus determining the shapes of the Okounkov bodies is an important task. Even in the surface case, there are still many questions that await to be answered.

The purpose of this paper is twofold. First, we completely determine the limiting and valuative Okounkov bodies of pseudoeffective divisors with respect to an arbitrary admissible flag on surfaces using the Zariski decompositions. As consequences, we show that the geometric properties of the given divisor and the admissible flag are reflected in the Okounkov bodies. Then, we try to find a chamber decomposition of the big cone such that the shape of the Okounkov bodies associated to the divisors in each chamber is constant.


Throughout the paper, by a surface $S$, we mean a smooth projective surface defined over an algebraically closed field of characteristic zero.
When $D$ is a big divisor on a surface, \cite[Theorem B]{KLM} completely characterize the Okounkov body of $D$. Our first main result is an extension of \cite[Theorem B]{KLM} to the pseudoeffective case.

\begin{theorem}[=Theorem \ref{compoklim} and Theorem \ref{exslope}]\label{maincomp}
Let $D$ be a non-big pseudoeffective divisor on a smooth projective surface $S$, and $D=P+N$ be the Zariski decomposition. Fix an admissible flag $C_\bullet : \{ x \} \subseteq C \subseteq S$. Then
$$
\oklim_{C_\bullet}(D)=\oklim_{C_\bullet}(P) + (\mult_C N, \ord_x ((N-(\mult_C N)C)|_C)),
$$
and $\oklim_{C_\bullet}(P)$ is given as follows:
\begin{enumerate}[\indent$(1)$]
 \item Suppose that $P.C>0$. Then $C$ is a positive volume subvariety of $D$ and $\kappanu(D)=1$. Furthermore, we have
$$
\oklim_{C_\bullet}(P) = \{ (0, x_2) \mid 0 \leq x_2 \leq P.C \}.
$$
Hence, $\dim ( \oklim_{C_\bullet}(D)) = 1$ and $\vol_{\R^1}(\oklim_{C_\bullet}(D)) = \vol_{S|C}^+(D) = P.C$.

 \item Suppose that $P.C=0$. Let $\mu:=\mu(D;C)$ be the Nakayama constant of $D$ along $C$. If $\mu>0$, then $\kappa_{max}(D) \geq 0$, and we can write $P \equiv \mu C + N'$ for some effective divisor $N'$. In this case, we have
$$
\begin{array}{c}
\oklim_{C_\bullet}(P) = \{ ( x_1, x_2) \mid 0 \leq x_1 \leq \mu, x_2 = \frac{\ord_x(N'|_C)}{\mu}x_1 \}
\end{array}
$$
Furthermore, $\dim ( \oklim_{C_\bullet}(D)) = 0$ if $\mu=0$ and $\dim ( \oklim_{C_\bullet}(D)) =1= \kappanu(D)$ if $\mu>0$.
\end{enumerate}

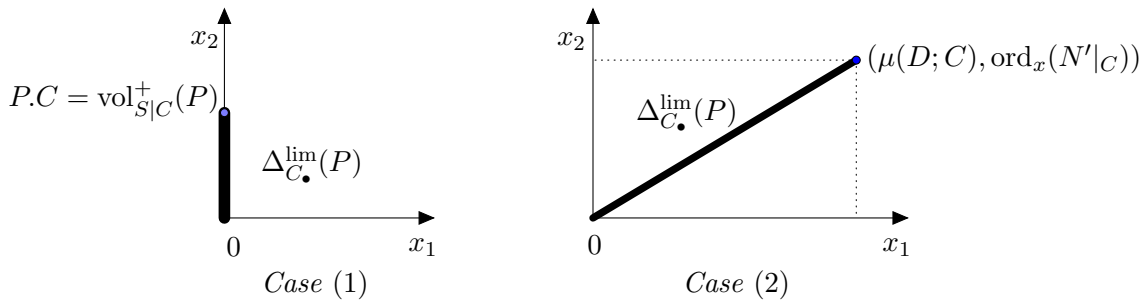
\begin{tikzpicture}[line cap=round,line join=round,>=triangle 45,x=0.7cm,y=0.7cm]
\clip(-4.8,-2.0) rectangle (24.66,4.5);
\draw [->] (0,0) -- (4,0);
\draw (-0.16,-0.18) node[anchor=north west] {$0$};
\draw (0.5,-0.8) node[anchor=north west] {Case $(1)$};
\draw (-0.84,3.74) node[anchor=north west] {$x_2$};
\draw (3.3,-0.2) node[anchor=north west] {$x_1$};
\draw [->] (7,0) -- (7,4);
\draw [->] (7,0) -- (13,0);
\draw (6.7,-0.14) node[anchor=north west] {$0$};
\draw (8.5,-0.8) node[anchor=north west] {Case $(2)$};
\draw (6.12,3.72) node[anchor=north west] {$x_2$};
\draw (12.3,-0.2) node[anchor=north west] {$x_1$};
\draw [->] (0,0) -- (0,4);
\draw [line width=4.4pt] (0,0)-- (0,2);
\draw [line width=2.8pt] (7,0)-- (12,3);
\draw [dotted] (12,3)-- (12,0);
\draw [dotted] (12,3)-- (7,3);
\draw (0.5,1.56) node[anchor=north west] {$\oklim_{C_{\bullet}}(P)$};
\draw (-4.3,2.8) node[anchor=north west] {$P.C=\mathrm{vol}^+_{S|C}(P)$};
\draw (7.62,2.54) node[anchor=north west] {$\oklim_{C_{\bullet}}(P)$};
\draw (12.0,3.5) node[anchor=north west] {$(\mu(D;C),\mathrm{ord}_x(N'|_C))$};
\begin{scriptsize}
\draw [fill=xdxdff] (0,2) circle (1.5pt);
\draw [fill=qqqqff] (12,3) circle (1.5pt);
\end{scriptsize}
\end{tikzpicture}
In particular, if $D$ is a $\Q$-divisor, then the limiting Okounkov body $\oklim_{C_\bullet}(D)$ is a line segment in $\R_{\geq 0}^2$ with a nonnegative rational slope with rational end points. Conversely, for any nonnegative rational number $r \in \Q_{\geq 0}$, there exists a smooth projective surface $S$, a pseudoeffective $\Q$-divisor $D$ on $S$, and an admissible flag $C_{\bullet}$ such that the limiting Okounkov body $\oklim_{C_\bullet}(D)$ has a slope $r$.
\end{theorem}

All the necessary notions are recalled in Sections \ref{sec-prem}, \ref{sec3}, and \ref{sec4}.
To prove Theorem \ref{maincomp}, we study basic properties of Nakayama constants, Zariski decompositions, and asymptotic base loci in Section \ref{sec4}. Note that all the cases in Theorem \ref{maincomp} do occur (see Example \ref{oklimex}).
As applications, we determine the infinitesimal limiting Okounkov body (Corollary \ref{infoklimdim}), and compute the Seshadri constant via the Okounkov body (Theorem \ref{sesh}).

We also study an analogous statement to Theorem \ref{maincomp} for the valuative Okounkov body $\okval_{C_\bullet}(D)$ of an effective divisor $D$ with respect to any admissible flag $C_\bullet$ on a surface $S$ (see Theorem \ref{compokval}).

Next, we study how the shapes of Okounkov bodies change as we vary the divisors. The following is the second main result of this paper.

\begin{theorem}[=Theorem \ref{thrm-main3}]\label{mainchamber}
Let $S$ be a smooth projective surface such that $\ol\Eff(X)$ is rational polyhedral, and fix an admissible flag $C_\bullet : \{ x \} \subseteq C \subseteq S$ where $C$ is a general member of the linear system of a very ample divisor on $S$ and $x$ is a general point in $C$.
Then the limiting Okounkov bodies $\oklim_{C_\bullet}(D_i)$ for all $D_i$ in a given Minkowski chamber $M$ are all similar.
\end{theorem}

See Section \ref{sec-prem} for the brief review on the decomposition of $\Bigdiv(S)$ into the stability chambers $\SC$ and Section \ref{sec5} for basic definitions of convex geometry.

The paper is organized as follows.
We start in Section \ref{sec-prem} by recalling basic notions and properties of divisors.
Section \ref{sec3} reviews the construction of the Okounkov body as in \cite{lm-nobody} and \cite{KK}, and presents the main results of \cite{chpw1} on the limiting and valuative Okounkov bodies.
In Section \ref{sec4}, we show Theorem \ref{maincomp} and give some applications to Nakayama constants and Seshadri constants.
Section \ref{sec5} is devoted to proving Theorem \ref{mainchamber}.

\subsection*{Aknowledgement}
We would like to thank the referee for helpful suggestions and comments.

\section{Preliminaries}\label{sec-prem}

Throughout the paper, $S$ denotes a smooth projective surface defined over an algebraically closed field of characteristic zero and $D$ denotes a pseudoeffective $\R$-divisor on $S$ unless otherwise stated. In this section, we briefly recall basic notions and properties which we need later on.

\subsection{Asymptotic base locus}

When $D$ is a $\Q$-divisor, we define the \emph{stable base locus} of $D$ as
$$
\text{SB}(D):= \bigcap_m \text{Bs}(mD)
$$
where the intersection is taken over all positive integers $m$ such that $mD$ are $\Z$-divisors, and $\text{Bs}(mD)$ denotes the base locus of the linear system $|mD|$.
The \emph{augmented base locus} of $D$ is defined as
$$
\bp(D):=\bigcap_A\text{SB}(D-A)
$$
where the intersection is taken over all ample divisors $A$ such that $D-A$ are $\Q$-divisors.
The \emph{restricted base locus} of $D$ is defined as
$$
\bm(D):=\bigcup_{A}\text{SB}(D+A)
$$
where the union is taken over all ample divisors $A$ such that $D+A$ are $\Q$-divisors.
We have $\bm(D) \subseteq \text{SB}(D) \subseteq \bp(D)$ for a $\Q$-divisor $D$.
One can check that a divisor $D$ is ample (or nef) if and only if $\bp(D)=\emptyset$
(respectively, $\bm(D)=\emptyset$). 
The asymptotic base loci $\bp(D)$ and $\bm(D)$ depend only on the numerical class of $D$.
For more details, see \cite{elmnp-asymptotic inv of base}, \cite{elmnp-restricted vol and base loci}.

\subsection{Volume of a divisor}

When $D$ is a $\Q$-divisor, the \emph{volume} of $D$ is defined as
$$
\vol_S(D):= \limsup_{m \to \infty} \frac{h^0(S, \mathcal{O}_S(mD))}{m^2/2!}.
$$
The volume $\vol_S(D)$ depends only on the numerical class of $D$.
Furthermore, this function uniquely extends to a continuous function
$$
\vol_S : \Bigdiv(S) \to \R.
$$
Note that if $D$ is not big (i.e., $S=\bp(D)$), then $\vol_S(D)=0$ . For more details, see \cite{pos}.

Let $V$ be a proper subvariety of $S$ such that $V\not\subseteq\bp(D)$.
If $\dim V=1$, then the \emph{restricted volume} of $D$ along $V$ is defined as
$$
\vol_{S|V}(D):=\limsup_{m \to \infty} \frac{h^0(S|V,mD)}{m}
$$
where $h^0(S|V,mD)$ is the dimension of the image of the natural restriction map $\varphi:H^0(S,\mc O_S(D))\to H^0(V,\mc O_V(D))$ (\cite[Definition 2.1]{elmnp-restricted vol and base loci}).
If $\dim V=0$, then we simply let $\vol_{S|V}(D)=1$.
The restricted volume $\vol_{X|V}(D)$ depends only on the numerical class of $D$.
Furthermore, this function uniquely extends to a continuous function
$$
\vol_{S|V} : \text{Big}^V (S) \to \R
$$
where $\text{Big}^V(S)$ is the set of all $\R$-divisor classes $\xi$ such that $V$ is not properly contained in any irreducible component of $\bp(\xi)$.
By \cite[Theorem 5.2]{elmnp-restricted vol and base loci}, if $V$ is an irreducible component of $\bp(D)$, then $\vol_{S|V}(D)=0$. For more details, see \cite{elmnp-restricted vol and base loci}.

Now let $V\subseteq S$ be a subvariety such that $V\not\subseteq\bm(D)$.
For an ample divisor $A$ on $S$, we define the \emph{augmented restricted volume} of $D$ along $V$ as
$$
\vol_{S|V}^+(D):=\lim_{\eps\to 0+} \vol_{S|V}(D+\eps A).
$$
The definition is independent of the choice of $A$.
As with $\vol_S$ and $\vol_{S|V}$, one can check that the augmented restricted volume $\vol_{S|V}^+(D)$ depends only on the numerical class of $D$.
By the continuity of the function $\vol_{S|V}$, we see that
$\vol_{S|V}^+(D)$ coincides with $\vol_{S|V}(D)$ for $D\in\Bigdiv^V(S)$.
For $D\in\Bigdiv^V(S)$, the following inequalities hold by definition:
$$\vol_{S|V}(D) \leq \vol_{S|V}^+(D) \leq \vol_{V}(D|_V).$$
See \cite{chpw1} for more properties of $\vol^+_{S|V}$.

\subsection{Iitaka dimension}

Let $\mathbb N(D)=\{m\in\Z_{>0}|\; |\lfloor mD\rfloor|\neq\emptyset\}$.
For $m\in\mathbb N(D)$, let $\Phi_{mD}:S\dashrightarrow \mathbb P^{\dim|\lfloor mD\rfloor|}$ be the rational map defined by the linear system $|\lfloor mD\rfloor |$.
We define the \emph{Iitaka dimension} of $D$ as the following value
$$
\kappa(D):=\left\{
\begin{array}{ll}
\max\{\dim\text{Im}(\Phi_{mD}) \mid \;m\in\mathbb N(D)\}& \text{if }\mathbb N(D)\neq\emptyset\\
-\infty&\text{if }\mathbb N(D)=\emptyset.
\end{array}
\right.
$$
Note that the Iitaka dimension $\kappa(D)$ depends on the linear equivalence class of $[D] \in \text{Pic}(S) \otimes \R$.
We also define the \emph{maximal Iitaka dimension} of $D$ as follows:
$$
\kappa_{max}(D):=\max \{\kappa(D') \mid D \equiv D' \}.
$$
By definition, $\kappa_{max}(D)$ depends only on the numerical class $[D]\in\N^1(X)_{\R}$.

Fix a sufficiently ample $\Z$-divisor $A$ on $S$. We define the \emph{numerical Iitaka dimension} of $D$ as the nonnegative integer
$$
\kappanu(D):= \max\left\{k \in \Z_{\geq0} \left|\; \limsup_{m \to \infty} \frac{h^0(S,\mathcal{O}_S(\lfloor mD \rfloor + A))}{m^k}>0 \right.\right\}
$$
if $h^0(S,\mathcal{O}_S(\lfloor mD \rfloor + A))\neq\emptyset$ for infinitely many $m>0$ and we let $\kappanu(D):=-\infty$ otherwise.
Remark that our $\kappanu$ is denoted by $\kappa_{\sigma}$ in \cite{lehmann-nu} and \cite{nakayama}.
The numerical Iitaka dimension $\kappanu(D)$ depends only on the numerical class $[D]\in\N^1(X)_{\R}$.
One can easily check that $\kappa(D) \leq \kappanu(D)$ holds and the inequality is strict in general (see \cite[Example 6.1]{lehmann-nu}).
However, if $\kappanu(D)=\dim X$, then $\kappa(D)=\dim X$.
See \cite{lehmann-nu} and \cite{nakayama} for detailed properties of $\kappa$ and $\kappanu$.

\subsection{Zariski decomposition}

Let $D$ be a pseudoeffective divisor on $S$. It is well known that $D$ admits the unique \emph{Zariski decomposition}: $D=P+N$, where the \emph{positive part} $P$ is nef, the \emph{negative part} $N=\sum a_iN_i$ is effective, $P.N_i=0$ for each irreducible component $N_i$, and the intersection matrix $(N_i.N_j)$ is negative definite if $N \neq 0$.
Note that $P$ is maximal in the sense that if $L$ is a nef divisor with $L \leq D$, then $L \leq P$.
If $D$ is a $\Q$-divisor, then so are the positive part $P$ and the negative part $N$.
The following is also well known (cf. \cite[Example 1.11]{elmnp-asymptotic inv of base}).

\begin{lemma}
Let $D=P+N$ be the Zariski decomposition of a pseudoeffective divisor $D$ on a surface $S$. Then $\bm(D)=\Supp(N)$. If we assume that $D$ is big, then
$\bp(D)=\Null(P)$.
\end{lemma}

Recall that the \emph{null locus} $\Null(P)$ of a nef and big divisor $P$ on a surface $S$ is the union of all irreducible curves $C$ on $S$ with $C.P=0$.

We now briefly recall the $s$-decomposition. For more details, we refer to \cite{P}.
Let $D$ be an effective $\Q$-divisor on $S$. We define
$$
N_s:=\inf\{L \mid L \sim_{\Q} D, L \geq 0\} \text{ and } P_s := D-N_s.
$$
Then we can check that $P_s$ and $N_s$ are $\Q$-divisors.
The expression $D=P_s+N_s$ is the \emph{$s$-decomposition} of $D$.
We note that $P_s$ is the minimal in the sense that if $L$ is an effective divisor with $H^0(S, mL) \simeq H^0(S, mP_s)$ for all sufficiently divisible integers $m>0$, then $P_s \leq L$.

\subsection{Chamber decomposition of the big cone}
We recall the chamber decomposition of the big cone $\Bigdiv(S)$ in the sense of \cite{BKS-Zar chamb}. Using the Zariski decomposition, we can define the following chambers in $\Bigdiv(X)$.

\begin{definition}\label{def-chambers}
Let $D$ be a big divisor on a surface $S$ and $D=P_D+N_D$ its Zariski decomposition.
\begin{enumerate}
\item We define the \emph{Zariski chamber} (associated to a nef divisor $P$) as $$\Sigma_P:=\{D\in\Bigdiv(S)|\;\Supp(N_D)=\Null(P) \}.$$

\item We define the \emph{Stability chamber} (associated to a big divisor $D$) as
$$\SC(D):=\{D'\in\Bigdiv(S)|\;\bp(D)=\bp(D')\}.$$
\end{enumerate}
\end{definition}

By \cite[Theorem 2.2]{BKS-Zar chamb}, whenever $\Int\Sigma_P\cap\Int\SC(D)\neq\emptyset$,
we have $\Int\Sigma_P=\Int\SC(D)$. Thus the chamber decompositions of $\Bigdiv(S)$ into the Zariski chambers and stability chambers only differ in the boundaries of the chambers; the two decompositions are \emph{essentially} the same.

\begin{theorem}[{\cite[Main Theorem]{BKS-Zar chamb}}]\label{thrm-BKS decomposition}
The big cone $\Bigdiv(X)$ has a locally finite decomposition into the Zariski chambers $\Sigma_P$
(or equivalently into the stability chambers $\SC(D)$ by the above remark) that are locally rational polyhedral.
\end{theorem}

\begin{remark}
Let $D$ be a big divisor on a surface $S$ and $D=P_D+N_D$ its Zariski decomposition.
\begin{enumerate}
\item We note that all the stability chambers $\SC(D)$ intersect with the nef cone $\Nef(X)$ since $\bp(D)=\Null(P_D)=\bp(P_D)$.
    However, a Zariski chamber $\Sigma_P$ can be disjoint from the nef cone $\Nef(X)$ as can be checked in the example of  \cite[Example 3.5]{BKS-Zar chamb}.

\item We will see that the structure of the Okounkov body of $D$ descends to that of the positive part $P_D$. Thus by (1), to study the structure of the Okounkov bodies of the divisors in some stability chamber $\SC(D)$, it is actually enough to study the divisors in $\SC(D)\cap\Nef(S)$. We will clarify this in Section \ref{sec5}.
\end{enumerate}
\end{remark}

\section{Construction and basic properties of Okounkov bodies}\label{sec3}

In this section, we briefly recall the construction of the Okounkov bodies of big divisors in \cite{lm-nobody} and \cite{KK}, and review the main results of \cite{chpw1}. For simplicity, we only consider the surface case.
As before, let $S$ be a smooth projective surface defined over an algebraically closed field of characteristic zero.
We fix an \emph{admissible flag} on $S$
$$
C_{\bullet} : \{x \} \subseteq C \subseteq S
$$
where $C$ is an integral curve and $x$ is a smooth point of $C$.
For an effective Cartier divisor $D$ on $S$ and a section $s \in H^0(S,\mathcal{O}_S(D)) \setminus \{0\}$, we define the function
$$\nu(s) = \nu_{C_{\bullet}}(s):=(\nu_1(s), \nu_2(s)) \in \Z_{\geq 0}^2$$
as follows.
 First, let $\nu_1(s):=\ord_C (s)$. Using a local equation $f$ for $C$ in $S$, we define a section
$s'_1=s \otimes f^{-\nu_1(s)} \in H^0(S, \mathcal{O}_S(D-\nu_1(s)C))$. Since $s'_1$ does not vanish identically along $C$, its restriction $s'_1|_C$ defines a nonzero section $s_1:=s'_1|_C \in H^0(C, \mathcal{O}_C(D-\nu_1(s) c))$. Now take $\nu_2(s):=\ord_x(s_1)$.
Note that $\nu_2(s)$ does not depend on the choice of the local equation $f$.

\subsection{Okounkov bodies of big divisors}

Now assume that $D$ is a big divisor on $S$.
The \emph{Okounkov body} $\okbd_{C_\bullet}(D)$ of $D$ with respect to the admissible flag $C_\bullet$ is defined as the closure of the convex hull of $\nu_{C_\bullet}(|D|_{\R})$ in $\R^2_{\geq 0}$ where we set $|D|_{\R}:=\{ D' \mid D \sim_{\R} D' \geq 0 \}$.

\begin{theorem}[{\cite[Theorem A]{lm-nobody}}]
We have $\vol_{\R^2}(\okbd_{C_\bullet}(D)) = \frac{1}{2}\vol_S(D)$.
\end{theorem}

Note that if $D$ is not big, then $\vol_S(D)=0$.

\subsection{Okounkov bodies of pseudoeffective divisors}

When $D$ is only pseudoeffective, there are two natural ways to associate a lower dimensional convex body to $D$, which were introduced in \cite{chpw1}.

\begin{definition}
(1) When $D$ is effective, i.e., $|D|_{\R}\neq \emptyset$, the \emph{valuative Okounkov body} $\okval_{C_\bullet}(D)$ of $D$ with respect to the admissible flag $C_\bullet$ is defined as the closure of the convex hull of $\nu_{C_\bullet}(|D|_{\R})$ in $\R^2_{\geq 0}$. If $D$ is not effective, we define $\okval_{C_\bullet}(D)=\emptyset$.\\
(2) When $D$ is pseudoeffective, the \emph{limiting Okounkov body} $\oklim_{C_\bullet}(D)$ of $D$ with respect to the admissible flag $C_\bullet$ is defined as
$$
\oklim_{C_\bullet}(D):=\lim_{\epsilon \to 0+}\okbd_{C_\bullet}(D+\epsilon A) = \bigcap_{\epsilon >0} \okbd_{C_\bullet}(D+\epsilon A)
$$
where $A$ is an ample divisor on $S$. Note that $\oklim_{C_\bullet}(D)$ is independent of the choice of $A$.
If $D$ is not pseudoeffective, we define $\oklim_{C_\bullet}(D)=\emptyset$.
\end{definition}

By definition, it is easy to check that $\okval_{C_\bullet}(D) \subseteq \oklim_{C_\bullet}(D)$.

\begin{proposition}[{\cite[Proposition 3.3 and Lemma 4.8]{B}}]\label{dimok}
We have the following:
\begin{enumerate}[\indent$(1)$]
 \item $\dim \okval_{C_\bullet}(D) = \kappa(D)$.
 \item $\kappa_{max} (D) \leq \dim \oklim_{C_\bullet}(D) \leq \kappanu(D)$.
\end{enumerate}
\end{proposition}

We now present examples for which the both inequalities in Proposition \ref{dimok} (2) are strict.

\begin{example}\label{ex-1}
We use Mumford's example (\cite[Example 1.5.2]{pos}) to show that $\dim \oklim_{C_\bullet}(D)$ does not coincide with $\kappa_{max}(D)$ in general.

There exists a ruled surface $S=\P(E)$ such that $H:=\mathcal{O}_{\P(E)}(1)$ is nef, $\kappa_{max}(H)=0$, and $\kappanu(H)=1$. Let $F$ be a fiber. Then the nef cone and the pseudoeffective cone of $S$ coincide and it is generated by $H$ and $F$. For any irreducible curve $C$ on $S$, we may write $C \simeq aH+bF$ for some rational numbers $a, b \geq 0$. We fix an admissible flag
$$
C_\bullet : \{ x\} \subseteq C \subseteq S
$$
where $x$ is any smooth point on $C$.
If $b=0$, then $a>0$ and using Theorem \ref{maincomp}, we obtain
$$
\oklim_{C_\bullet}(H) = \left\{ (x_1, 0) \in \R^2 \left|\; 0 \leq x_1 \leq \frac{1}{a} \right.\right\}
$$
so that $\dim \oklim_{C_\bullet}(H)=1$.
If $b>0$, then
$$
\oklim_{C_\bullet}(H) = \{ (0, x_2) \in \R^2 \mid  0 \leq x_2 \leq b \}
$$
so that $\dim \oklim_{C_\bullet} (H)=1$.
\end{example}

\begin{example}\label{infokex}
Here we give an example such that $\dim \oklim_{C_\bullet}(D) < \kappanu(D)$.
Let $S$ be the surface as in Example \ref{ex-1} and $\pi : \widetilde{S} \to S$ be the blow-up at any point $y\in S$ with the exceptional divisor $E$.
We fix an admissible flag
$$
C_\bullet: \{ x \} \subseteq E \subseteq \widetilde{S}
$$
where $x$ is a general point on $E$.
Then using Theorem \ref{maincomp}, we can easily see that $\oklim_{C_\bullet} (\pi^*H)=\{ (0,0) \}$ even though $\kappanu(\pi^*H)=\kappanu(H)=1$.
\end{example}

Actually, $\oklim_{C_\bullet} (\pi^*H)$ is the infinitesimal limiting Okounkov body $\oklim_{\inf}(H)$ which we will define below in Definition \ref{infokdef}. We will show that $\dim \oklim_{\inf}(D) = \max\{ 0, \kappa_{\max}(D) \}$ for the surface case (see Corollary \ref{infoklimdim}).

\begin{definition}\label{infokdef}
Let $\pi : \widetilde{S} \to S$ be the blow-up at a point $y\in S$ with the exceptional divisor $E$. We consider an admissible flag
$$
C_\bullet: \{ x \} \subseteq E \subseteq \widetilde{S}
$$
where $x$ is a point on $E$.
If $y\in S$ and $x\in E$ are chosen very generally, then the Okounkov body $\okbd_{C_\bullet}(\pi^*D)$ is called
the \emph{infinitesimal Okounkov body} of $D$ and we denote it by $\okbd_{\inf}(D)$.
Similarly, if $D$ is a pseudoeffective divisor on $S$, then the Okounkov body $\oklim_{C_\bullet}(\pi^*D)$ for the general choices of $y\in S$ and $x\in E$ is called the
\emph{infinitesimal limiting Okounkov body} and is denoted by $\oklim_{\inf}(D)$.
\end{definition}

If $D$ is a big divisor on $S$ and we choose $y\in S$ and $x\in E$ very generally, then all $\okbd_{C_\bullet}(\pi^*D)$ coincide by
\cite[Proposition E]{lm-nobody}. Similarly, even in the case where $D$ is pseudoeffective, we can also easily check that $\oklim_{C_\bullet}(\pi^*D)$ all coincide. Thus the above definitions of $\okbd_{\inf}(D)$ and $\oklim_{\inf}(D)$ are well-defined.

In \cite{chpw1}, the following two special subvarieties were introduced and studied.

\begin{definition}
Let $D$ be a divisor on a surface $S$.
\begin{enumerate}[\indent$(1)$]
\item For an effective divisor $D$ on $S$, a smooth subvariety $U \subseteq S$  is called a \emph{Nakayama subvariety} of $D$ if $\kappa(D)=\dim U$ and the natural map
$$
H^0(S, \mc O_S(\lfloor mD \rfloor)) \to H^0(U, \mc O_U(\lfloor mD|_U \rfloor))
$$
is injective (or equivalently, $H^0(S, \mc I_U \otimes \mc O_S(\lfloor mD \rfloor))=0$ where $\mc I_U$ is the ideal sheaf of $U$ in $S$) for every integer $m \geq 0$.

\item For a pseudoeffective divisor $D$ on $S$, a subvariety $V\subseteq X$ of dimension $\kappanu(D)$
such that $\vol_{S|V}^+(D)>0$ and $V\not\subseteq\bm(D)$ is called a \emph{positive volume subvariety} of $D$.
\end{enumerate}
\end{definition}

By definition, $U=S$ is the Nakayama subvariety (or positive volume subvariety) of $D$ if and only if $D$ is big.
It is proven that any general subvariety $U \subseteq S$ of dimension $\kappa(D)$ is a Nakayama subvariety of $D$.
Similarly, any general subvariety $V \subseteq S$ of dimension $\kappanu(D)$ is a positive volume subvariety of $D$.

In \cite{chpw1}, the following were shown:
\begin{enumerate}[\indent$(1)$]
\item If an admissible flag  $C_\bullet : \{x\} \subseteq C \subseteq S$ contains a Nakayama subvariety $U$ of $D$ and $x$ is a general point, then $\okval_{C_\bullet}(D) \subseteq \{ 0 \}^{2-\kappa(D)} \times \R_{\geq 0}^{\kappa(D)}$ so that we can regard $\okval_{C_\bullet}(D) \subseteq \R_{\geq 0}^{\kappa(D)}$.
\item  If an admissible flag  $C_\bullet : \{x\} \subseteq C \subseteq S$ contains a positive volume subvariety $V$ of $D$, then $\oklim_{C_\bullet}(D) \subseteq \{ 0 \}^{2-\kappanu(D)} \times \R_{\geq 0}^{\kappanu(D)}$ so that we can regard $\oklim_{C_\bullet}(D) \subseteq \R_{\geq 0}^{\kappanu(D)}$.
\end{enumerate}

\begin{theorem}[{\cite[Theorems A and B]{chpw1}}]\label{chpwthm}
Let $D$ be a pseudoeffective divisor on a surface $S$, and fix an admissible flag $C_\bullet : \{x\} \subseteq C \subseteq S$.
We have the following:
\begin{enumerate}[\indent$(1)$]
 \item Suppose that $D$ is effective, the admissible flag $C_\bullet$ contains a Nakayama subvariety $U$ of $D$, and $x$ is a general point. Then
$$
\dim \okval_{C_\bullet} (D) = \kappa(D) \text{ and } \vol_{\R^{\kappa(D)}} (\okval_{C_\bullet} (D)) = \vol_{S|U}(D).
$$
 \item Suppose that the admissible flag  $C_\bullet : \{x\} \subseteq C \subseteq S$ contains a positive volume subvariety $V$ of $D$. Then
$$
\dim \oklim_{C_\bullet}(D) = \kappanu(D) \text{ and } \vol_{\R^{\kappanu(D)}} (\oklim_{C_\bullet}(D)) = \vol_{S|V}^+(D).
$$
\end{enumerate}
\end{theorem}

\section{Limiting Okounkov bodies of pseudoeffective divisors on surfaces}\label{sec4}

The aim of this section is to give an explicit description of the Okounkov bodies of pseudoeffective divisors on a smooth surface. We first review the known properties of the Okounkov bodies of big divisors, and give simple proofs which also work for the limiting Okounkov bodies of pseudoeffective divisors. Next, we prove the main results Theorems \ref{compoklim} and \ref{exslope}, and present some applications.
As before, $S$ denotes a smooth projective surface defined over an algebraically closed field of characteristic zero.

\subsection{Nakayama constant and Zariski decomposition}

For a pseudoeffective divisor $D$ on $S$ and a subvariety $V$ of $S$, we define the \emph{Nakayama constant of $D$ along $V$} as
 $$
\mu(D; V) := \sup \{ s \geq 0 \mid f^*D-sE \text{ is pseudoeffective} \}
$$
where $f : \widetilde{S} \to S$ is the blow-up of $S$ at $V$ with the exceptional divisor $E$.
Note that if $V$ is an integral curve, then we take $f=id$, $\widetilde{S} = S$ and $E=V$.

The Nakayama constant and Zariski decomposition play an important role in studying the Okounkov body as in the following theorem.

\begin{theorem}[{\cite[Theorem 6.4]{lm-nobody} and \cite[Theorem 4.4]{chpw1}}]\label{surfokbd}
Let $D$ be a pseudoeffective divisor on a surface $S$. Fix an admissible flag $C_\bullet : \{ x \} \subseteq C \subseteq S$.
Let $a:=\mult_C N$ where $D=P+N$ is the Zariski decomposition, and $\mu :=\mu(D;C)$. Consider the divisor $D_t:= D-tC$ for $a \leq t \leq \mu$. Denote by $D_t=P_t+N_t$ the Zariski decomposition. Let $\alpha(t):=\ord_x(N_t|_C)$ and $\beta(t):=\alpha(t)+C.P_t$. Then the limiting Okounkov body $\oklim_{C_\bullet}(D)$ of $D$ is given by
$$
\oklim_{C_\bullet}(D)=\{(x_1, x_2) \mid a \leq x_1 \leq \mu \text{ and } \alpha(x_1) \leq x_2 \leq \beta(x_1) \}.
$$
\end{theorem}

Now we show some basic properties of the Nakayama constant and the Zariski decomposition on a surface.

\begin{lemma}\label{nakp}
Let $D$ be a pseudoeffective divisor on a surface $S$, and $D=P+N$ be the Zariski decomposition. For an integral curve $C$, we have $\mu(D;C)=\mu(P;C)+\mult_C N$.
\end{lemma}

\begin{proof}
By replacing $D$ by $D- \left( \mult_C N \right) C$, we can assume that $C$ is not in the support of $N$. Then we only have to show that $\mu(D;C)=\mu(P;C)$.
Note that $\mu(D;C) \geq \mu(P;C)$. Thus it is sufficient to show that if $D-tC$ is pseudoeffective for some $t \geq 0$, then so is $P-tC$. Let $D-tC=P_t + N_t$ be the Zariski decomposition. Then
$$
P+N=D=P_t + N_t + tC.
$$
By the maximal property of the positive part of the Zariski decomposition, we obtain $N \leq N_t + tC$. Since $C$ is not an irreducible component of $N$, the divisor $N_t-N$ is effective. Then $P-tC = P_t + (N_t-N)$ is pseudoeffective as desired.
\end{proof}

The following was first established in \cite[Proposition 2.1]{KLM} for big divisors.

\begin{corollary}
    Let $D$ be a pseudoeffective divisor on $S$ with the Zariski decomposition $D=P+N$, and $C$ be an integral curve in $S$. Assume that $C$ is not an irreducible component of $N$.
For $t_1 > t_2 \geq 0$, assume that $D-t_1 C$ is pseudoeffective so that we have the Zariski decompositions $D-t_1C = P_{t_1} + N_{t_1}$ and $D-t_2C=P_{t_2} + N_{t_2}$. Then $N_{t_1} \geq N_{t_2}$.
\end{corollary}

\begin{proof}
The assertion was already shown in the proof of Lemma \ref{nakp}.
\end{proof}

Next, we show the rationality of the Nakayama constant of a non-big pseudoeffective divisor.

\begin{lemma}\label{nakrat}
Let $D$ be a pseudoeffective $\Q$-divisor on a surface $S$, and $D = P+N$ be the Zariski decomposition. If $D$ is not big, then $\mu(D; C)$ is a rational number for any integral curve $C$ in $S$.
\end{lemma}

\begin{proof}
By Lemma \ref{nakp}, we have $\mu(D; C) = \mu(P; C) + \mult_C N$. Thus it suffices to show that $\mu:=\mu(P; C)$ is a rational number. Let $P-\mu C = P_{\mu} + N_{\mu}$ be the Zariski decomposition. Suppose that $P.C>0$. By Theorem \ref{surfokbd}, we get
$$
\oklim_{C_\bullet}(P) \supseteq \{ (0, x_2) \mid 0 \leq x_2 \leq P.C \}
$$
where $C_\bullet: \{x \} \subseteq C \subseteq S$ and $x \in C$ is any point. Since $P$ is not big,
it follows that $\mu=0$. It remains to consider the case $P.C=0$. Then we have $P.P_{\mu}=0$. By the Hodge index theorem, we have $P_{\mu}=kP$ for some $k \geq 0$. By the definition of the Nakayama constant, we get $k=0$. Thus $P = N_{\mu} + \mu C$. We can conclude that $\mu$ is a rational number.
\end{proof}

\begin{remark}
If $D$ is big, then $\mu(D;C)$ is a rational number or satisfies a quadratic equation over $\Q$ by \cite[Proposition 2.2]{KLM}.
\end{remark}

We further study some easy properties of the Zariski decompositions of divisors of the form $P-tC$.

\begin{lemma}\label{ptlem}
Let $D$ be a pseudoeffective divisor on $S$ with the Zariski decomposition $D=P+N$, and $C$ be an integral curve in $S$. Assume that $C$ is not an irreducible component of $N$.
For $t>0$, assume that $D-tC$ is pseudoeffective so that we have the Zariski decomposition $D-tC = P_t + N_t$. Then we have the following:
\begin{enumerate}[\indent$(1)$]
 \item $C$ is not a component of $N_t$.
 \item If $E$ is an integral curve such that $P.E=0$, $E^2<0$, and $E \neq C$, then $P_t + (N_t+sE)$ is the Zariski decomposition for $s \geq 0$.
\end{enumerate}
\end{lemma}

\begin{proof}
(1) If $C$ is a component of $N_t$, then $D = P_t + (N_t + tC)$ is the Zariski decomposition. However, $N \neq N_t + tC$, so we get a contradiction.

\noindent (2) Note that $(P-tC).E \leq 0$. Thus $N_t.E \leq 0$, so either $E$ is an irreducible component of $N_t$ or $E$ does not meet $N_t$. For the latter case, we have $P_t.E=N_t.E=0$. Thus in any case, we obtain $P_t.(N_t+E)=0$ and  the intersection matrix of $N_t+sE$ is negative definite.
\end{proof}

The following is well known.

\begin{lemma}\label{dtoplem}
Let $D$ be a pseudoeffective divisor on $S$, and $D=P+N$ be the Zariski decomposition. Fix an admissible flag $C_\bullet : \{ x \} \subseteq C \subseteq S$. Then we have
$$
\oklim_{C_\bullet}(D)=\oklim_{C_\bullet}(P) + (\mult_C N, \ord_x ((N-(\mult_C N)C)|_C)).
$$
\end{lemma}

\begin{proof}
First, consider the case $C \subseteq \bm(D)$. Set $a:=\mult_C N>0$. It is easy to see that
$$
\oklim_{C_\bullet}(D) = \oklim_{C_\bullet} (D-aC) + (a,0).
$$
By replacing $D$ by $D-aC$, we may only consider the case $C \not\subseteq \bm(D)$. Then it is sufficient to show that
\begin{equation}\tag{!}\label{okdvsp}
\oklim_{C_\bullet}(D) = \oklim_{C_\bullet}(P) + (0, \ord_x (N|_C)).
\end{equation}
Fix $t >0$ such that $P-tC$ is pseudoeffective. Let $P-tC = P_t + N_t$ be the Zariski decomposition. By Lemma \ref{ptlem}, $P+N-tC = P_t + (N_t+N)$ is the Zariski decomposition. Then the assertion (\ref{okdvsp}) now follows from Theorem \ref{surfokbd}.
\end{proof}

\begin{remark}
We can easily verify that a similar statement of Lemma \ref{dtoplem} holds for the valuative Okounkov body $\okval_{C_\bullet}(D)$ of an effective divisor $D$. Let $D=P_s+N_s$ be the $s$-decomposition. Then we have
$$
\okval_{C_\bullet}(D)=\okval_{C_\bullet}(P_s) + (\mult_C N_s, \ord_x ((N_s-(\mult_C N_s)C)|_C)).
$$
\end{remark}

\subsection{Asymptotic base loci via Okounkov bodies}

Here we give simpler proofs for the following two theorems.
These were first shown in \cite[Theorem A]{AV-loc pos} for big divisors.

\begin{theorem}\label{bm}
Let $D$ be a pseudoeffective divisor on a surface $S$. Then the following are equivalent:
\begin{enumerate}[\indent$(1)$]
 \item $x \in \bm(D)$.
 \item For every flag $C_\bullet : \{x\} \subseteq C \subseteq S$, the limiting Okounkov body $\oklim_{C_\bullet}(D)$ does not contain the origin of $\R^2$.
 \item For some flag $C_\bullet : \{x\} \subseteq C \subseteq S$, the limiting Okounkov body $\oklim_{C_\bullet}(D)$ does not contain the origin of $\R^2$.
\end{enumerate}
\end{theorem}

\begin{proof}
Let $D = P+N$ be the Zariski decomposition.

\noindent$(1)\Rightarrow(2)$: Since $\ord_x (N|_C) >0$, the assertion follows from Theorem \ref{surfokbd}.

\noindent$(2)\Rightarrow(3)$: It is obvious.

\noindent$(3)\Rightarrow(1)$: By Theorem \ref{surfokbd}, we have either $\mult_C N >0$ or $\ord_x (N|_C) >0$. In both cases, an irreducible component of $N$ passes through $x$, so $x \in \bm(D)$.
\end{proof}

\begin{remark}\label{9ptsex}
We cannot replace $\oklim_{C_\bullet}(D)$ by $\okval_{C_\bullet}(D)$. For an explicit example, we consider the blow-up $\pi : S \to \P^2$ of $\P^2$ at 9 general points on a cubic curve $C$ in $\P^2$. Note that $-K_S = \pi^{-1}_* C$ is nef and $\kappa(-K_S)=0$. Consider an admissible flag $C_{\bullet} : \{x \} \subseteq \pi^{-1}_*C \subseteq S$, where $x$ is any smooth point in $\pi^{-1}_*C $. Then we can easily see that
$$
\okval_{C_\bullet}(-K_S) = \{ (1, 0) \}
$$
which does not contain the origin even though $-K_S$ is nef. However, we have
$$
\oklim_{C_\bullet}(-K_S) = \{(x_1, 0) \mid 0 \leq x_1 \leq 1 \}
$$
which contains the origin of $\R^2$.
\end{remark}

\begin{theorem}
Let $D$ be a pseudoeffective divisor on a surface $S$. Then the following are equivalent:
\begin{enumerate}[\indent$(1)$]
 \item $x \in \bp(D)$.
 \item For every flag $C_\bullet : \{x\} \subseteq C \subseteq S$, the limiting Okounkov body $\oklim_{C_\bullet}(D)$ does not contain $U \cap \R_{\geq 0}^2$ where $U$ is a small open neighborhood of the origin of $\R^2$.
 \item For some flag $C_\bullet : \{x\} \subseteq C \subseteq S$, the limiting Okounkov body $\oklim_{C_\bullet}(D)$ does not contain $U \cap \R_{\geq 0}^2$ where $U$ is a small open neighborhood of the origin of $\R^2$.
\end{enumerate}
\end{theorem}

\begin{proof}
If $D$ is pseudoeffective but not big, then $\bp(D)=S$ and $\dim \oklim_{C_\bullet}(D) < 2$ for any flag $C_\bullet$. In this case, there is nothing to prove.
Thus we only have to consider the case where $D$ is big. Let $D = P+N$ be the Zariski decomposition.

\noindent$(1)\Rightarrow(2)$: By considering Theorem \ref{bm}, we can assume that $x \in \bp(D) \setminus \bm(D)$. By Lemma \ref{dtoplem}, we obtain $\oklim_{C_\bullet}(D) = \oklim_{C_\bullet}(P)$.
We divide into two cases. First, consider the case $C \subseteq \bp(P) = \text{Null}(P)$, i.e., $P.C=0$. In this case, $\oklim_{C_\bullet}(P)$ does not meet the $x_2$-axis by Theorem \ref{surfokbd}. More precisely, for any $(0, y)$ with $y>0$, we have $(0, y) \not\in \oklim_{C_\bullet}(P)$. Now, consider the remaining case $C \not\subseteq \bp(P)$. We can take an integral curve $E$ such that $x \in E \subseteq \bp(P)=\text{Null}(P)$. Note that $C\neq E$ but both $C$ and $E$ contain $x$. Thus $C.E >0$ so that $(P-tC).E<0$ for all $t>0$.
Since $P$ is big, $P-t_0C$ is pseudoeffective for some $t_0 >0$.
Let $P-t_0C = P_{t_0} + N_{t_0}$ be the Zariski decomposition. Then $N_{t_0}.E<0$ so that $E$ is an irreducible component of $N_{t_0}$. Thus $\ord_x(N_{t_0}|_C) >0$. In view of Theorem \ref{surfokbd}, $\oklim_{C_\bullet}(D)$ does not meet the $x_1$-axis. That is, for any $(y,0)$ with $y>0$, we have $(y, 0) \not\in \oklim_{C_\bullet}(D)$.

\noindent$(2)\Rightarrow(3)$: It is obvious.

\noindent$(3)\Rightarrow(1)$: Suppose that $x \not\in \bp(D)$. For an ample divisor $A$ and a sufficiently small $\epsilon >0$, we have $\bp(D) = \bm(D-\epsilon A)$. By Theorem \ref{bm}, the origin is contained in $\oklim_{C_\bullet}(D-\epsilon A)$ for any admissible flag $C_\bullet$.
Now we have
$$
\oklim_{C_\bullet}(D-\epsilon A ) + \oklim_{C_\bullet}(\epsilon A) \subseteq \oklim_{C_\bullet}(D).
$$
Since $\oklim_{C_\bullet}(\epsilon A)$ contains $U \cap \R_{\geq 0}^2$, so does $\oklim_{C_\bullet}(D)$.
\end{proof}

\subsection{Computing limiting Okounkov bodies}

We now prove the main results of this section. When $D$ is a big divisor on a surface $S$, Theorem \ref{surfokbd} and \cite[Theorem B]{KLM} completely characterize the Okounkov body $\okbd_{C_\bullet}(D)$ of $D$ with respect to any admissible flag $C_\bullet$.
Our main results, Theorems \ref{compoklim} and \ref{exslope}, can be regarded as a natural extension of \cite[Theorem B]{KLM} to the case of pseudoeffective divisors.

\begin{theorem}\label{compoklim}
Let $D$ be a non-big pseudoeffective divisor on a surface $S$, and $D=P+N$ be the Zariski decomposition. Fix an admissible flag $C_\bullet : \{ x \} \subseteq C \subseteq S$. Then
$$
\oklim_{C_\bullet}(D)=\oklim_{C_\bullet}(P) + (\mult_C N, \ord_x ((N-(\mult_C N)C)|_C)),
$$
and $\oklim_{C_\bullet}(P)$ is given as follows:
\begin{enumerate}[\indent$(1)$]
 \item Suppose that $P.C>0$. Then $C$ is a positive volume subvariety of $D$ and $\kappanu(D)=1$. Furthermore, we have
$$
\oklim_{C_\bullet}(P) = \{ (0, x_2) \mid 0 \leq x_2 \leq P.C \}.
$$
Hence, $\dim ( \oklim_{C_\bullet}(D)) = 1$ and $\vol_{\R^1}(\oklim_{C_\bullet}(D)) = \vol_{S|C}^+(D) = P.C$.

 \item Suppose that $P.C=0$. Let $\mu:=\mu(D;C)$. If $\mu>0$, then $\kappa_{max}(D) \geq 0$, and we can write $P \equiv \mu C + N'$ for some effective divisor $N'$. In this case, we have
$$
\begin{array}{c}
\oklim_{C_\bullet}(P) = \{ ( x_1, x_2) \mid 0 \leq x_1 \leq \mu, x_2 = \frac{\ord_x(N'|_C)}{\mu}x_1 \}
\end{array}
$$
Furthermore, $\dim ( \oklim_{C_\bullet}(D)) = 0$ if $\mu=0$ and $\dim ( \oklim_{C_\bullet}(D)) =1= \kappanu(D)$ if $\mu>0$.
\end{enumerate}
\end{theorem}

\begin{proof}
By Lemma \ref{dtoplem}, we can assume that $D=P$.
By Theorem \ref{bm}, the origin of $\R^2$ is contained in $\oklim_{C_\bullet}(D)$.
If $P.C>0$, then the assertion immediately follows from Theorem \ref{chpwthm} and Theorem \ref{surfokbd} (see also \cite{chpw1}). It remains to consider the case $P.C=0$. If $\mu=0$, then $\oklim_{C_\bullet}(D)$ is the origin of $\R^2$, and there is nothing to prove. We now suppose that $\mu>0$. Let $P-\mu C = P_{\mu}+N_{\mu}$ be the Zariski decomposition. We claim that $P_{\mu}=0$. If this claim holds, then the remaining assertion is a direct consequence of Theorem \ref{surfokbd}. We have $0=P.(P-\mu C)=P.P_{\mu}+P.N_{\mu}$. Since $P^2=P.C=0$, it follows that $P.P_{\mu}=0$. By the Hodge index theorem, $P_{\mu}=kP$ for some $k \geq 0$. By the definition of the Nakayama constant, we obtain $k=0$, so we are done.
\end{proof}

\begin{remark}
For the case (2) of Theorem \ref{compoklim}, the volume of the limiting Okounkov body $\oklim_{C_\bullet}(D)$ is $\sqrt{\mu^2 +\ord_x(N'|_C)^2 }$.
However, the geometric meaning is not clear to us.
\end{remark}

\begin{remark}
Using Theorem \ref{compoklim}, one can easily check that $\vol^+_{S|C}(D)=\vol^+_{S|C}(P)=P.C$. Thus $C$ is a positive volume subvariety of $D$ if and only if $P.C>0$.
\end{remark}

Next examples show that all the cases in Theorem \ref{compoklim} do occur.

\begin{example}\label{oklimex}
(1) By \cite{chpw1}, there always exists a positive volume subvariety of any pseudoeffective divisor, so the first case of Theorem \ref{compoklim} does occur.\\
(2) We give examples of the second case of Theorem \ref{compoklim} with $\mu=0$.
For any flag $C_\bullet : \{x\} \subseteq C \subseteq S$, we have $\oklim_{C_\bullet}(\mathcal{O}_S)=\{ (0,0) \}$. In this case, $\kappanu(\mathcal{O}_S)=\kappa(\mathcal{O}_S)=0$. On the other hand, Example \ref{infokex} gives an example of $\oklim_{C_\bullet}(D) = \{ (0,0) \}$, but $\kappanu(D)=1$. \\
(3) Remark \ref{9ptsex} gives an example of the second case of Theorem \ref{compoklim} with $\mu>0$ and a horizontal limiting Okounkov body. In this case, note that $\kappa(-K_S)=0$.\\
(4) For an example of the second case of Theorem \ref{compoklim} with $\mu>0$ and a limiting Okounkov body with a positive slope, consider a fibration  $f: S \to C$ onto a curve $C$. Assume that there exists a fiber $F$ of $f$ such that we can write
$$
F=pC_1 + qC_2 + E
$$
where $C_1$ and $C_2$ are integral curves transversally meeting at a point $x$ and $E$ is an effective divisor whose support contains $x$, but does not contain neither $C_1$ nor $C_2$.

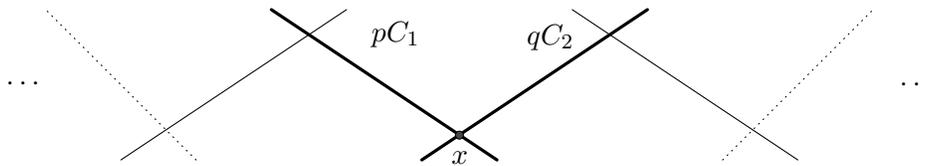
\begin{tikzpicture}[line cap=round,line join=round,>=triangle 45,x=1.0cm,y=1.0cm]
\clip(-4.3,-0.7) rectangle (24.66,2.7);
\draw (-1,0)-- (2,2);
\draw [line width=1.2pt] (1,2)-- (4,0);
\draw (5,2)-- (8,0);
\draw [dotted] (0,0)-- (-2,2);
\draw [line width=1.2pt] (6,2)-- (3,0);
\draw [dotted] (7,0)-- (9,2);
\draw (-2.66,1.24) node[anchor=north west] {$\cdots$};
\draw (9.22,1.22) node[anchor=north west] {$\cdots$};
\draw (3.26,0.25) node[anchor=north west] {$x$};
\draw (2.2,1.98) node[anchor=north west] {$pC_1$};
\draw (4.26,1.94) node[anchor=north west] {$qC_2$};
\begin{scriptsize}
\draw [fill=uuuuuu] (3.5,0.33) circle (1.5pt);
\end{scriptsize}
\end{tikzpicture}
For the existence of such a fibration, see the proof of Theorem \ref{exslope}.
Consider the admissible flags $C_{1\bullet} : \{x\} \subseteq C_1 \subseteq S$ and $C_{2\bullet}: \{x\} \subseteq C_2 \subseteq S$. Then we can see that
$$
\oklim_{C_{1\bullet}}(F) = \{(x_1, x_2) \mid 0 \leq x_1 \leq p, x_2 = \frac{q}{p}x_1 \} \text{ and }
\oklim_{C_{2\bullet}}(F) = \{(x_1, x_2) \mid 0 \leq x_1 \leq q, x_2 = \frac{p}{q}x_1 \}.
$$
In this case, note that $\kappa(F)=1$.
\end{example}

By Lemma \ref{nakrat} and Theorem \ref{compoklim}, the limiting Okounkov body $\oklim_{C_\bullet}(D)$ is a line segment in $\R^2$ with a rational slope when $D$ is a $\Q$-divisor.
We show that the converse of this statement also holds.

\begin{theorem}\label{exslope}
Let $r\in\Q_{\geq0}$ be any nonnegative rational number. Then there exist a smooth projective surface $S$, a pseudoeffective $\Q$-divisor $D$ on $S$, and an admissible flag $C_\bullet : \{ x \} \subseteq C \subseteq S$ such that the limiting Okounkov body $\oklim_{C_\bullet}(D)$ has a slope $r$.
\end{theorem}

\begin{proof}
By Example \ref{oklimex}, we only have to deal with the case $r=\frac{p}{q}>0$ with relatively prime positive integers $p$ and $q$. It suffices to show the existence of a fibration $f: S \to C$ such that a fiber $F$ of $f$ can be written as
$$
F=pC_1 + qC_2 + E
$$
where $C_1$ and $C_2$ are integral curves transversally meeting at a point $x$ and $E$ is an effective divisor whose support does not contain neither $C_1$ nor $C_2$. For this purpose, we first consider $\P^1 \times \P^1$ with a fibration $\P^1 \times \P^1 \to \P^1$. For any integer $m>0$, by taking a successive blow-ups of $\P^1 \times \P^1$, we can make one fiber contain two irreducible components with multiplicities $1$ and $m$ transversally meeting at a point.
Suppose that we have one fiber containing two irreducible components with multiplicities $m$ and $n$ transversally meeting at a point. Then by taking a  successive blow-ups of that surface, we can obtain a fiber containing two irreducible components with multiplicities $m+kn$ and $n$ for any integer $k>0$ transversally meeting at a point.
By considering the Euclidean algorithm for $p$ and $q$, we can take a successive blow-ups of $\P^1 \times \P^1$ such that the resulting surface has a fiber containing two irreducible components with multiplicities $p$ and $q$ transversally meeting at a point.
\end{proof}

\begin{remark}
It is shown in \cite[Theorem B]{KLM} that any real polygon satisfying some conditions in $\R_{\geq 0}^2$ can be realized as the Okounkov body of a big divisor on a smooth projective toric surface.
The example given in the proof of Theorem \ref{exslope} is also a smooth projective toric surface.
\end{remark}

For the valuative Okounkov body $\okval_{C_\bullet}(D)$ of an effective divisor $D$, the analogous statement to Theorem \ref{compoklim} also holds. Since the proof is similar to that of Theorem \ref{compoklim} and Theorem \ref{exslope}, we omit it here. Recall that we always have $\dim \okval_{C_\bullet}(D)=\kappa(D)$ by Proposition \ref{dimok} (1).

\begin{theorem}\label{compokval}
Let $D$ be an effective divisor on a surface $S$, and $D=P_s+N_s$ be the $s$-decomposition.
Assume that $D$ is not big. Fix an admissible flag $C_\bullet : \{ x \} \subseteq C \subseteq S$. Then
$$
\okval_{C_\bullet}(D)=\okval_{C_\bullet}(P_s) + (\mult_C N_s, \ord_x ((N_s-(\mult_C N_s)C)|_C)),
$$
and $\okval_{C_\bullet}(P_s)$ is given as follows:
\begin{enumerate}[\indent$(1)$]
 \item Suppose that $P_s.C>0$. Then $\kappa(D)=1$, and we have
$$
\okval_{C_\bullet}(P_s) = \{ (x_1, x_2) \mid 0 \leq x_2 \leq \vol_{S|C}(D) \}.
$$

 \item Suppose that $P_s.C=0$. Let $\mu:=\mu(D;C)$. If $\mu=0$, then $\okval_{C_\bullet}(P_s)=\{ (0,0)  \}$ and $\kappa(D)=0$.
If $\mu>0$, then $\kappa(D)=1$ and we can write $P_s \sim \mu C + N'$ for some effective divisor $N'$. In this case, we have $\okval_{C_\bullet}(P_s)=\oklim_{C_\bullet}(P_s)$.
\end{enumerate}
In particular, if $D$ is a $\Q$-divisor, then the valuative Okounkov body $\okval_{C_\bullet}(D)$ is a line segment in $\R_{\geq 0}^2$ with a nonnegative rational slope with rational end points. Conversely, for any nonnegative rational number $r \in \Q_{\geq 0}$, there exists a smooth projective surface $S$, an effective $\Q$-divisor $D$ on $S$, and an admissible flag $C_\bullet$ such that the valuative Okounkov body $\okval_{C_\bullet}(D)$ has a slope $r$.
\end{theorem}

As the first application of our main results, we can completely understand the infinitesimal limiting Okounkov
body of a pseudoeffective divisor.

\begin{corollary}\label{infoklimdim}
Let $D$ be a non-big pseudoeffective divisor on a surface $S$.
Let $x\in S$ be a general point. Then we have
$$
\oklim_{\inf}(D) = \{(x_1, 0) \mid 0 \leq x_1 \leq \mu (D;x) \}.
$$
Hence, $\dim \oklim_{\inf}(D) = \max\{ 0, \kappa_{\max}(D) \}$ and $\vol_{\R^1} \oklim_{\inf}(D) = \mu (D;x)$.
\end{corollary}

\begin{proof}
Let $D = P+N$ be the Zariski decomposition, and $f : \widetilde{S} \to S$ be the blow-up at a general point $x\in S$ with the exceptional divisor $E$. Then $f^*D = f^*P + f^*N$ is the Zariski decomposition.
Note that $f^*P.E=0$.
By the generality assumption, we may assume that $x$ is not contained in the support of $N$. Thus $E$ is not a component of $f^*N$.
By Lemma \ref{dtoplem}, we have $\oklim_{\inf}(D)=\oklim_{\inf}(P)$.
Consequently, $\oklim_{\inf}(D)$ contains the origin of $\R^2$ by Theorem \ref{bm}.
If $\kappa_{max}(D)=-\infty$ or $0$, then $\oklim_{\inf}(D)=\{(0,0)\}$. If $\kappa_{max}(D)=1$, then by the generality assumption and Theorem \ref{compoklim}, the assertion follows.
\end{proof}

\subsection{Seshadri constant}

Finally, we compute the Seshadri constant via the limiting Okounkov body.
For a nef divisor $D$ and a subvariety $V$, we define the \emph{Seshadri constant of $D$ along $V$} as follows:
$$
\epsilon(D; V) := \sup \{ s \geq 0 \mid f^*D-sE \text{ is nef} \}
$$
where $f: \widetilde{S} \to S$ is the blow-up of $S$ at $V$ with the exceptional divisor $E$.
Note that if $V$ is an integral curve, then we take $f=id$, $\widetilde{S}=S$ and $E=V$.
 We can compute the Seshadri constant along an integral curve by using the limiting Okounkov bodies.

\begin{theorem}\label{sesh}
Let $D$ be a nef divisor and $C$ be a smooth curve on $S$. Then we have
$$
\epsilon (D;C)=\inf_{x \in C} \{ s \mid (s, 0) \not\in \oklim_{C_\bullet}(D) \text{ where } C_\bullet: \{x\} \subseteq C \subseteq S \text{ is an admissible flag} \}.
$$
\end{theorem}

\begin{proof}
For $0 \leq t \leq \mu(D;C)$, let $D-tC = P_t + N_t$ be the Zariski decomposition.
If $N_t=0$ for all $0 \leq t \leq \mu(D;C)$, then the assertion is trivial. Thus we now assume that $N_s \neq 0$ for some $0 < s \leq \mu(D;C)$.
Note that
$$
\epsilon:=\epsilon (D;C) = \inf \{s \mid N_s \neq 0 \}.
$$
We denote by
$$
\epsilon' := \inf_{x \in C} \{ s \mid (s, 0) \not\in \oklim_{C_\bullet}(D) \text{ where } C_\bullet: \{x\} \subseteq C \subseteq S \text{ is an admissible flag} \}.
$$
For $0 \leq s \leq \mu(D;C)$, it is enough to show that $(s, 0) \not\in \oklim_{C_\bullet}(D)$ for some admissible flag $C_\bullet: \{x\} \subseteq C \subseteq S$ and if and only if $N_s \neq 0$.
If $(s, 0) \not\in \oklim_{C_\bullet}(D)$, then by Theorem \ref{surfokbd}, $\ord_x (N_s|_C) >0$ so that $N_s \neq 0$.
For the converse, we suppose that $N_s \neq 0$. By Lemma \ref{ptlem} (1), $C$ is not a component of $N_s$. Suppose that $C.N_s = 0$, i.e., $C$ does not meet any irreducible component of $N$. Since the intersection matrix of irreducible components of $N_s$ is negative definite, there is an effective divisor $E$ such that $\text{Supp}(E) \subseteq \text{Supp}(N_s)$ and $N_s.E < 0$. Then we obtain
$$
D.E = (P_s+N_s+sC).E = N_s.E<0,
$$
which is a contradiction. Thus $C.N_s>0$, so $C$ meets $N_s$ at some point $x$. For the admissible flag $C_\bullet: \{x\} \subseteq C \subseteq S$, we get $(s, 0) \not\in \oklim_{C_\bullet}(D)$. Hence we are done.
\end{proof}



\section{Okounkov bodies on chambers}\label{sec5}

In this section, we study how the shape of the Okounkov body $\okbd_{C_\bullet}(D)$ changes as we vary $D$.
We first need to clarify what we mean by saying that $\oklim_{C_\bullet}(D)$ and $\oklim_{C_\bullet}(D')$ have the \emph{same} shape.

\begin{definition}\label{def-same okbd}
Let $\Delta, \Delta'\subseteq\R^2$ be convex rational polytopes.
We say $\Delta$ and $\Delta'$ are \emph{similar} and write $\Delta\approx \Delta'$ if $\Delta,\Delta'$ have the same number of vertices $\{v_1,\cdots,v_m=v_0\}$, $\{w_1,\cdots,w_m=w_0\}$, and edges $\{v_iv_{i+1}\}$, $\{w_iw_{i+1}\}$, respectively, that can be labeled in such a way that the rays $\overrightarrow{v_iv_{i+1}}$ and $\overrightarrow{w_iw_{i+1}}$ are parallel for all $i$.
\end{definition}

Two rays $\overrightarrow{V_1}, \overrightarrow{V_2}$ in $\R^2$ emitting from the points $O_1, O_2$ respectively are \emph{parallel} if the translated rays $\overrightarrow{V_1}-O_1$ and $\overrightarrow{V_2}-O_2$ coincide.
Note that a finite sequence of rays defines a $\approx$-equivalence class of polytopes in $\R^2$.

We define the \emph{Minkowski sum} of two subsets $\okbd,\okbd'\subseteq\R^2$ as
$$
\okbd+\okbd':=\{\mathbf x+\mathbf x'|\;\mathbf x\in\okbd,\;\mathbf x'\in\okbd'\}.
$$
We say that a convex bodies $\okbd$ is \emph{indecomposable}
if $\okbd=\okbd_1+\okbd_2$ for convex bodies $\okbd_1,\okbd_2$ implies $\okbd_1=a_1\okbd$
and $\okbd_2=a_2\okbd$ where $a_1,a_2\geq 0$ and $a_1+a_2=1$.
Note that the line segments and simplices are the only indecomposable convex rational polytopes in $\R^2$.

\begin{lemma}\label{lem-minko sum}
Let $\mc M=\{\okbd_1,\cdots,\okbd_m\}$ be a finite set of indecomposable convex rational polytopes of $\R^2$.
Then the Minkowski sums $\sum_{i=1}^m a_i\okbd_i$ for all $a_i>0$ are similar to each other.
\end{lemma}

\begin{proof}
We proceed the induction on $m$. The assertion is trivial if $m=1$. Assume that $m \geq 2$. By the induction hypothesis, $\sum_{i=1}^{m-1} a_i \okbd_i$ for all $a_i >0$ are similar to each other. Thus it is sufficient to show that if $\okbd$ and $\okbd'$ are similar rational convex polytopes and $\okbd''$ is an indecomposable convex rational polytope in $\R^2$, then $\okbd+\okbd''$ and $\okbd'+\okbd''$ are similar.
It is easy to check that the numbers of vertices of $\okbd+\okbd''$ and $\okbd'+\okbd''$ are the same. Furthermore, the Minkowski sum $\okbd+\okbd''$ (resp. $\okbd'+\okbd''$) is a convex polytope whose sides consist of the sides of $\okbd$ and $\okbd''$ (resp. $\okbd'$ and $\okbd''$). Thus $\okbd+\okbd''$ and $\okbd'+\okbd''$ are similar.
\end{proof}

We now consider the Minkowski decomposition of a divisor. For more details, we refer to \cite{PD}.
Let $S$ be a smooth projective surface such that $\ol\Eff(X)$ is rational polyhedral, and fix an admissible flag $C_\bullet : \{ x \} \subseteq C \subseteq S$ where $C$ is a general member of the linear system of a very ample divisor on $S$ and $x$ is a general point in $C$.
Let $D$ be a pseudoeffective divisor on $S$, and $D=P+N$ be the Zariski decomposition. By Lemma \ref{dtoplem}, we have $\oklim_{C_\bullet}(D)=\oklim_{C_\bullet}(P)$.
Thus it is enough to consider nef divisors.
By \cite[Main Theorem]{PD}, there exists a finite set $\mc M$ (which is called the \emph{Minkowski basis} with respect to $C_\bullet$) of nef $\Q$-divisors such that for any nef divisor $D$, we have
$$
D = \sum_{B_i \in \mc M} b_i B_i ~~\text{ and } ~~\oklim_{C_\bullet}(D)=\sum_{B_i \in \mc M}b_i \oklim_{C_\bullet}(B_i)
$$
where all $\oklim_{C_\bullet}(P)$ are indecomposable.
The presentation $D = \sum_{B_i \in \mc M} b_i B_i$ is called the \emph{Minkowski decomposition} of $D$ with respect to $\mc M$.

We recall the construction of the Minkowski basis $\mc M$ with respect to $C_\bullet$ (see \cite[Section 3.1]{PD}). First, the generators of extremal rays of $\Nef(S)$ belong to $\mc M$. Additionally, for each stability chamber $\SC$, we include in $\mc M$ the corresponding Minkowski basis element $B$ as follows. Let $N_1, \ldots, N_k$ be integral curves in the support of $\bp(D)$ for any $D \in \SC$. Then there is the unique nef divisor $B=C + \sum_{i=1}^k n_i N_i$ such that $n_i \geq 0$ and $B.N_i=0$ for all $1 \leq i \leq k$.
We also briefly explain how to obtain the Minkowski decomposition of a nef divisor $D$ (see \cite[Section 3.2]{PD}). If $D$ is not big, then we can write $D = \sum b_i B_i$ where $B_i$ are generators of extremal rays of the face of $\Nef(S)$ containing $D$. Since $C$ is ample, it follows from Theorem \ref{compoklim} that $\oklim_{C_\bullet}(D)$ is a vertical line segment of length $D.C$ in the $x_2$-axis and $\oklim_{C_\bullet}(B_i)$ are vertical line segments of length $B_i.C$ in the $x_2$-axis. Thus $\oklim_{C_\bullet}(D)=\sum b_i \oklim_{C_\bullet}(B_i)$, and hence, $D = \sum b_i B_i$ is the Minkowski decomposition. If $D$ is big, then we consider the stability chamber $\SC(D)$. Let $B_D$ be the corresponding Minkowski basis element to $\SC(D)$, and $b_{B_D}:=\sup \{ s \geq 0 \mid D-s B_D \text{ is nef} \}$. Then $D-b_{B_D} B_D$ is a nef divisor and lies in a face of the closure $\overline{\SC(D)}$, and $\oklim_{C_\bullet}(D)=b_{B_D}\oklim_{C_\bullet}(B_D) + \oklim_{C_\bullet}(D-b_{B_D} B_D)$. By continuing this process, we finally obtain the Minkowski decomposition of $D$.

Now we define the Minkowski chamber decomposition of the nef cone $\Nef(X)$ with respect to $C_\bullet$ following \cite{SS}. For a Minkowski basis element $B$ which is not in any of the extremal rays of $\Nef(S)$, we can decompose $\Nef(S)$ into the subcones $M_i$ generated by the extremal rays of $\Nef(S)$ and the ray spanned by $B$.
If $B'$ is another Minkowski basis element  which is not in any of the extremal rays of $\Nef(S)$, then we can decompose further into the subcones generated by the extremal rays of $M_i$ and the ray spanned by $B'$.
Repeat the process with all the Minkowski basis elements not in the extremal rays of $\Nef(S)$.
The interior of each subcone in the decomposition of $\Nef (S)$ we obtain at the end is called the \emph{Minkowski chamber} of $\Nef(S)$.

\begin{example}
Let $f : S \to \P^2$ be the blow-up of two general points in $\P^2$ with exceptional divisors $E_1, E_2$, and $H:=f^* L$ where $L$ is a line in $\P^2$.
Note that the nef cone $\Nef(S)$ is generated by $H, H-E_1, H-E_2$.
If $C \in |3H-E_1-E_2|$ is a general member, then $\{ H, H-E_1, H-E_2, 2H-E_1-E_2, 3H-E_1, 3H-E_2, 3H-E_1-E_2\}$ is a Minkowski basis with respect to $C_\bullet$ and the Minkowski chamber decomposition is given in the picture on the right below. 
If $C \in |H|$ is a general member, then $\{ H, H-E_1, H-E_2, 2H-E_1-E_2 \}$ is a Minkowski basis with respect to $C_\bullet$ and the Minkowski chamber decomposition is given as the picture on the right below.

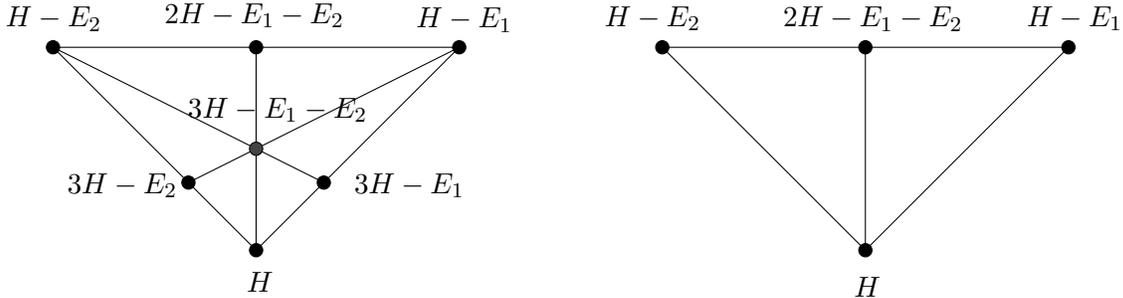
\definecolor{uuuuuu}{rgb}{0.26666666666666666,0.26666666666666666,0.26666666666666666}
\begin{tikzpicture}[line cap=round,line join=round,>=triangle 45,x=0.09cm,y=0.09cm]
\clip(-8,-10.) rectangle (160.,40.);
\draw (0.,30.)-- (60.,30.);
\draw (30.,0.)-- (60.,30.);
\draw (30.,0.)-- (0.,30.);
\draw (30.,30.)-- (30.,0.);
\draw (0.,30.)-- (40.,10.);
\draw (20.,10.)-- (60.,30.);
\draw (-8.4752437487342815,37.5318080240546) node[anchor=north west] {$H-E_2$};
\draw (14.9624349234463,37.99424081066438) node[anchor=north west] {$2H-E_1-E_2$};
\draw (52.17858572392799,37.37766376185133) node[anchor=north west] {$H-E_1$};
\draw (0.602152965323393,12.872601480175115) node[anchor=north west] {$3H-E_2$};
\draw (42.997313601227916,12.872601480175115) node[anchor=north west] {$3H-E_1$};
\draw (27.124330002934364,-1.754526380321706) node[anchor=north west] {$H$};
\draw (18.420992301413683,24.04612278559039) node[anchor=north west] {$3H-E_1-E_2$};
\draw (90.,30.)-- (150.,30.);
\draw (90.,30.)-- (120.,0.);
\draw (150.,30.)-- (120.,0.);
\draw (120.,30.)-- (120.,0.);
\draw (80.00356275593792,37.58153917049374) node[anchor=north west] {$H-E_2$};
\draw (106.36610700133808,37.51910638388395) node[anchor=north west] {$2H-E_1-E_2$};
\draw (142.34910370419368,37.57325064608721) node[anchor=north west] {$H-E_1$};
\draw (116.83629060523266,-2.3876804779478) node[anchor=north west] {$H$};
\begin{scriptsize}
\draw [fill=black] (0.,30.) circle (2.5pt);
\draw [fill=black] (60.,30.) circle (2.5pt);
\draw [fill=black] (30.,0.) circle (2.5pt);
\draw [fill=black] (30.,30.) circle (2.5pt);
\draw [fill=black] (40.,10.) circle (2.5pt);
\draw [fill=black] (20.,10.) circle (2.5pt);
\draw [fill=uuuuuu] (30.,15.) circle (2.5pt);
\draw [fill=black] (90.,30.) circle (2.5pt);
\draw [fill=black] (120.,30.) circle (2.5pt);
\draw [fill=black] (150.,30.) circle (2.5pt);
\draw [fill=black] (120.,0.) circle (2.5pt);
\end{scriptsize}
\end{tikzpicture}
\end{example}

\begin{lemma}\label{minkowski}
Let $S$ be a smooth projective surface such that $\ol\Eff(X)$ is rational polyhedral, and fix an admissible flag $C_\bullet : \{ x \} \subseteq C \subseteq S$ where $C$ is a general member of the linear system of a very ample divisor on $S$ and $x$ is a general point in $C$. For a given Minkowski chamber $M$, let $B_1, \ldots, B_k$ be the Minkowski basis elements in the closure $\overline{M}$. Then for any $D \in M$, we have the Minkowski decomposition $D = \sum_{i=1}^k b_i B_i$ such that all $b_i >0$.
\end{lemma}

\begin{proof}
The assertion follows from the construction of the Minkowski chambers.
\end{proof}

The following is the main result of this section.

\begin{theorem}\label{thrm-main3}
Let $S$ be a smooth projective surface such that $\ol\Eff(X)$ is rational polyhedral, and fix an admissible flag $C_\bullet : \{ x \} \subseteq C \subseteq S$ where $C$ is a general member of the linear system of a very ample divisor on $S$ and $x$ is a general point in $C$.
Then the limiting Okounkov bodies $\oklim_{C_\bullet}(D_i)$ for all $D_i$ in a given Minkowski chamber $M$ are all similar.
\end{theorem}

\begin{proof}
It follows from Lemmas \ref{lem-minko sum} and \ref{minkowski}.
\end{proof}


\begin{thebibliography}{ELMNP1}


\bibitem[B]{B} S. Boucksom, \textit{Corps D'Okounkov (d'apr\`{e}s Okounkov, Lazarsfeld- Musta\c{t}\u{a} et Kaveh-Khovanskii)}, S\'{e}minaire Bourbaki, 65\`{e}me ann\'{e}e, 2012-2013, no. 1059.



\bibitem[BKS]{BKS-Zar chamb} T. Bauer, A. K$\ddot{\text{u}}$ronya, and T. Szemberg,
\textit{Zariski chambers, volumes and stable base loci},
J. reine angew. Math. \textbf{576} (2004), 209-233.


\bibitem[CHPW]{chpw1} S. Choi, Y. Hyun, J. Park, and  J. Won,
\textit{Okounkov bodies associated to pseudoeffective divisors}, preprint,
arXiv:1508.03922.





\bibitem[ELMNP1]{elmnp-asymptotic inv of base} L. Ein, R. Lazarsfeld, M. Musta\c{t}\u{a}, M. Nakamaye, and M. Popa,
\textit{Asymptotic invariants of base loci}, Ann. Inst. Fourier \textbf{56} (2006), 1701-1734.

\bibitem[ELMNP2]{elmnp-restricted vol and base loci} L. Ein, R. Lazarsfeld, M. Musta\c{t}\u{a}, M. Nakamaye, and M. Popa, \textit{Restricted volumes and base loci of linear series}, Amer. J. Math. \textbf{131} (2009), 607-651.



\bibitem[KK]{KK} K. Kaveh and A. G. Khovanskii, \textit{Newton convex bodies, semigroups of integral points,
graded algebras and intersection theory}, Ann. of Math. (2) \textbf{176} (2012), 925-978.

\bibitem[KL]{AV-loc pos} A.  K\"{u}ronya and V. Lozovanu,
\textit{Local positivity of linear series on surfaces}, preprint,
arXiv:1411.6205.

\bibitem[KLM]{KLM} A. K\"{u}ronya, V. Lozovanu, and C. Maclean, \textit{Convex bodies appearing as Okounkov bodies of divisors}. Adv. Math. \textbf{229} (2012), 2622-2639.

\bibitem[La]{pos} R. Lazarsfeld, \textit{Positivity in algebraic geometry I \& II}, Ergeb. Math. Grenzgeb. \textbf{48} \& \textbf{49} (2004), Springer-Verlag, Berlin.

\bibitem[LM]{lm-nobody} R. Lazarsfeld and M. Musta\c{t}\u{a},
\textit{Convex bodies associated to linear series.}
Ann. Sci. Ec. Norm. Super. (4) \textbf{42} (2009), 783-835.

\bibitem[Le]{lehmann-nu} B. Lehmann,
\textit{Comparing numerical dimensions},
Alg. and Num. Theory. \textbf{7} (2013), 1065-1100.

\bibitem[LS]{PD} P. {\L}uszcz-\'{S}widecka and D. Schmitz \textit{Minkowski decomposition of Okounkov bodies on surfaces}, J. Algebra \textbf{414} (2014), 159-174.

\bibitem[N]{nakayama} N. Nakayama, \textit{Zariski-decomposition and abundance},
MSJ Memoirs \textbf{14}. Mathematical Society of Japan, Tokyo, 2004.


\bibitem[O1]{O1} A. Okounkov, \textit{Brunn-Minkowski inequality for multiplicities}, Invent. Math. \textbf{125} (1996) 405-411.

\bibitem[O2]{O2} A. Okounkov, \textit{Why would multiplicities be log-concave?} in \textit{The Orbit Method in Geometry and Physics}, Progr. Math. \textbf{213} (2003), Birkhäuser Boston, Boston, MA, 329-347.

\bibitem[P]{P} Y. Prokhorov,
\textit{On the Zariski decomposition problem}, Tr. Mat. Inst. Steklova \textbf{240} (2003) no. Biratsion. Geom. Linein. Sist. Konechno Porozhdennye Algebry, 43-72 (Russian, with Russian summary); English transl., Proc. Steklov Inst. Math. \textbf{1 (240)} (2003), 37-65.

\bibitem[SS]{SS} D. Schmitz and H. Sepp\"{a}nen, \textit{On the polyhedrality of global Okounkov bodies}, Adv. Geom. \textbf{16} (2016), 83-91.


\end{thebibliography}
\end{document}